\documentclass[12pt]{article}
\usepackage{eurosym}
\usepackage{amsfonts}
\usepackage{amsmath,amssymb,amsfonts}
\usepackage{amsfonts}
\usepackage{graphicx}
\usepackage{graphics}
\usepackage{url,hyperref}
\usepackage{xcolor}

\newtheorem{theorem}{Theorem}

\newtheorem{lemma}[theorem]{Lemma}

\newtheorem{proposition}[theorem]{Proposition}
\newtheorem{remark}[theorem]{Remark}

\newenvironment{proof}[1][Proof]{\noindent\textbf{#1.} }{\ \rule{0.5em}{0.5em}}

\begin{document}

\title{Outer billiards in the complex hyperbolic plane}
\author{Yamile Godoy and Marcos Salvai~ \thanks{%
This work was supported by Secretar\'{\i}a de Ciencia y T\'ecnica de la Universidad
Nacional de C\'ordoba and Consejo Nacional de Investigaciones Cient\'ificas
y T\'ecnicas; in particular, the PIBAA awarded to the first author by the last institution.}}
\date{ }
\maketitle

\begin{abstract}
Given a quadratically convex compact connected oriented hypersurface $N$ of the complex
hyperbolic plane, we prove that the characteristic rays of the symplectic
form restricted to $N$ determine a double geodesic foliation of the exterior 
$U$ of $N$. This induces an outer billiard map $B$ on $U$. We prove that $B$
is a diffeomorphism (notice that weaker notions of strict convexity may allow the
billiard map to be well-defined and invertible, but not smooth) and
moreover, a symplectomorphism. These results generalize known geometric
properties of the outer billiard maps in the hyperbolic plane and complex Euclidean space.
\end{abstract}

 \medskip


\section{Introduction}

Given a smooth, closed, strictly convex curve $\gamma $ in the plane, the 
\textbf{dual} or \textbf{outer billiard map} $B$ associated with $\gamma $
is defined as follows: Let $p$ be a point outside of $\gamma $. There are
two tangent lines to $\gamma $ through $p$; choose the right one from the
viewpoint of $p$, and define $B(p)$ as the reflection of $p$ in the point of
tangency (see Figure \ref{fig:outerplane}).

\begin{figure}[ht!]
	\label{fig:outerplane} 
	\centerline{
		\includegraphics[width=2in]
		{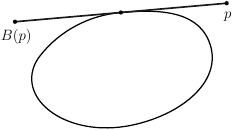}
	}
	\caption{The outer billiard map in the plane}
\end{figure}

The dual billiard was introduced by B.\ Neumann \cite{Neumann} and
popularized by J.\ Moser \cite{moser1, moser2}, who considered it as a toy
model for celestial mechanics. The outer billiard map has since been studied
in a number of settings; see \cite{TabachMathIntell, TabachnikovOuter,
	TabachnikovBilliards, bookT} for surveys.

In \cite{Tabachnikov}, Tabachnikov generalizes the outer billiard map in the
plane to the standard symplectic space $(\mathbb{R}^{2n},\omega )$ as
follows: Given a quadratically convex compact smooth hypersurface $M$ in $%
\mathbb{R}^{2n}$, the restriction of $\omega $ to each tangent space $T_{q}M$
has a one-dimensional kernel, called the characteristic line. Tabachnikov
proves that for each point $p$ outside $M$, exactly two characteristic lines
of $M$ pass through $p$. Thus, the choice of an orientation induces a
well-defined outer billiard map. Moreover, he proves that the map is a
symplectomorphism. A far reaching generalization of the results has been
obtained in the recent paper \cite{TabachNuevo}. On the other hand, in \cite%
{TabachHypPlane} Tabachnikov studies the outer billiard map in the
hyperbolic plane $H^{2}$, and proves in particular that it preserves the
area associated to the hyperbolic metric. Both in the symplectic and
hyperbolic cases, he addresses dynamical problems.

Outer billiards maps have been recently studied in other ambient spaces. For
instance, in \cite{London}, an outer billiard map is defined in the space of
oriented geodesics $\mathcal{G}$ of a three-dimensional space form and set
in the context of \cite{Tabachnikov} (using that $\mathcal{G}$ is a K\"{a}%
hler manifold). Also, following the approach of \cite{Tabachnikov}, in which
tangent rays in the characteristic directions 
double foliate the exterior of the strictly convex
hypersurface, in \cite{Ibero} the authors find conditions for unit tangent
vector fields on a complete umbilic not totally geodesic hypersurface of a
space form to determine an outer billiard map.

The goal of the present article is to generalize the above-mentioned results
in \cite{Tabachnikov} to the simplest curved analogue of $\mathbb{C}^{2}$,
that is, the complex hyperbolic plane $\mathbb{C}H^{2}$: We define and give
geometrical properties of an outer billiard system in this K\"{a}hler
two-point homogeneous space. It also generalizes the real hyperbolic
two-dimensional case, since $H^{2}$ is the complex hyperbolic line.

Let $N$ be a compact connected oriented hypersurface in $\mathbb{C}H^{2}$
which is quadratically convex, that is, the shape operator at each point is
definite. In particular, by \cite{Alex}, %(see also \cite{Lima}), 
it is the border of a convex body and diffeomorphic to a sphere. We call $U$
the exterior of $N$, that is, the connected component of $\mathbb{C}H^{2}-N$
into which the opposite of the mean curvature field points.

We call $J$ the complex structure of $\mathbb{C}H^{2}$. Given a unit tangent
vector $u$ we denote by $\gamma _{u}$ the unique geodesic of $\mathbb{C}%
H^{2} $ whose initial velocity is $u$. Let $F:N\times \mathbb{R}\rightarrow 
\mathbb{C}H^{2}$ be the smooth map defined by 
\begin{equation*}
	F(q,r)=\gamma _{J\nu (q)}(r)\text{,}
\end{equation*}%
where $\nu $ is the unit normal vector field pointing inwards, that is, in
the direction of the mean curvature vector field of $N$. This is our first
result:

\begin{theorem}
	\label{F+F-} Let $N$ be a quadratically convex compact connected oriented hypersurface of $%
	\mathbb{C}H^{2}$. Then, the restrictions of $F$ to $N\times (0,\infty )$ and 
	$N\times (-\infty ,0)$ are diffeomorphisms onto the exterior $U$ of $N$.
\end{theorem}

Thus, we have a \textbf{double geodesic ray foliation} of $U$ and this
allows us to define well an \textbf{outer billiard map}, as follows: For $%
t>0 $, \label{symplectic}%
\begin{equation*}
	B:U\rightarrow U,\hspace{1cm}B(\gamma _{J\nu (p)}(-t))=\gamma _{J\nu (p)}(t)%
	\text{.}
\end{equation*}

As a direct consequence of Theorem \ref{F+F-}, we have:

\begin{theorem}
	\label{Bdifeomorph}The outer billiard map $B$ is a diffeomorphism.
\end{theorem}

\begin{remark}
	In general, for a Hadamard manifold (a simply connected
	manifold with nonpositive sectional curvature), weaker hypotheses on the
	convexity of the hypersurface may allow the outer billiard map to be well
	defined without being smooth. For instance, Proposition 1.3 of \cite{London}
	presents an example of a strictly convex surface $S$ in $\mathbb{R}^{3}$
	(that is, for each $p\in S$, the affine plane tangent to $S$ at $p$
	intersects $S$ only at $p$, near $p$), for which the outer billiard map does
	exist, but it is not smooth.
\end{remark}

%\bigskip

The following proposition provides examples. Part (b) is a particular case
of part (a). We recall that a complex line in $\mathbb{C}H^{2}$ is a
complete totally geodesic surface which is a complex submanifold (its
tangent spaces are invariant by $J$). It is isometric to the real hyperbolic
plane $H^{2}\left( -4\right) $ of constant curvature $-4$. A complex line
admits an isometric reflection in $\mathbb{C}H^{2}$ with respect to it.

\begin{proposition}
	\label{Examples} \emph{a)} Let $H$ be a complex line in $\mathbb{C}H^{2}$ and let $N$
	be a quadratically convex compact connected oriented hypersurface invariant by the reflection
	with respect to $H$. Then $B$ preserves $H\cap U$ and coincides with the
	usual outer billiard map in $H$ associated with the 
	curve of positive geodesic curvature obtained by intersecting $N$ with $H.$
	
	\smallskip
	
	\emph{b)} Let $N$ be a geodesic sphere of radius $R$. Any complex line $H$ through
	the center $o$ of $N$ is preserved by the outer billiard map associated with 
	$N$. Thus, there is a $\mathbb{CP}^{1}$-worth of hyperbolic planes of
	constant curvature $-4$ through $o$ and the outer billiard is played outside
	the discs of radius $R$ centered at $o$ on each of them.
\end{proposition}

Notice that geodesic spheres in $\mathbb{C}H^{2}$, with the induced
Riemannian metric have not constant sectional curvature. They are Berger spheres;
they can be obtained by modifying in a suitable manner the lengths of the
fibers of the standard Hopf fibration $S^{3}\equiv SU\left( 2\right)
\rightarrow S^{2}$.

In the proof of Proposition \ref{Examples} we recall that isometric
reflections with respect a complex line $H$ do actually exist. In fact,
nontrivial isometric reflections are scarce, they exist only with respect
to complex lines or totally real planes, that is, the complete totally
geodesic surfaces of constant curvature $-4$ or $-1$, which are the only
totally geodesic surfaces.

\medskip

The question arises, whether a quadratically convex compact connected oriented hypersurface of $%
\mathbb{C}H^{2}$ such that all outer billiard trajectories are contained in
complex hyperbolic planes is necessarily a geodesic sphere. This subject was
addressed in \cite{Planar} in the Euclidean case.

\medskip

Next we present a theorem that generalizes
the geometric features (but not the dynamical ones) of \cite{Tabachnikov} and \cite{TabachHypPlane} (for the
complex Euclidean space and the hyperbolic plane, respectively) to the complex hyperbolic plane.

\begin{theorem}\label{TeoSymplectic}
	Let $N$ be a quadratically convex compact connected oriented hypersurface of $\mathbb{C}H^{2}$
	and let $U$ be the exterior of $N$. Then the outer billiard map $B$ is a
	symplectomorphism with respect to the K\"ahler form.
\end{theorem}

Regarding the dynamics of this system, an interesting problem would be to
find conditions on $N$ for the existence of three-periodic orbits.

For instance, if $N$ is a geodesic sphere of a small enough radius $R$, then
Proposition \ref{Examples} (b) implies a positive answer. In fact, the outer
billiard map is played outside discs of radius $R$ on hyperbolic planes of
constant curvature $-4$, so if $R$ is strictly smaller than the radius of the
inscribed circle of an ideal equilateral triangle, it is not difficult to
show that there are three-periodic orbits.

In the following we comment on the obstacles we found to address the problem
of the existence of three-periodic orbits for $N$ contained in a
sufficiently small ball, following Tabachnikov's arguments in \cite%
{TabachPeriodic}.

The main difficulty was that in $\mathbb{C}H^{2}$, although three distinct
points do determine a geodesic triangle, a distinguished triangular surface
dressing it exists (in contrast to $\mathbb{C}^{2}$) only if the three
points lie in a totally geodesic surface, and these are very scarce (just
totally real or complex surfaces).

Let $M$ be a Hadamard K\"{a}hler manifold with fundamental symplectic form $%
\omega $. Let $p_{1},p_{2},p_{3}$ be three points in $M$. The symplectic
area of the triangle determined by those points is well defined by%
\begin{equation*}
	\mathcal{A}\left( p_{1},p_{2},p_{3}\right) =\int_{\Delta }\omega \text{,}
\end{equation*}%
where $\Delta $ is any (suitably oriented) smooth surface whose boundary is
the triangle. Note that the integral does not depend on the choice of $%
\Delta $ since the form $\omega $ is closed.

Let $U$ be as above the exterior of $N$. To avoid considering the size of $N$%
, we focus on necessary conditions for $\left( z_{1},z_{2},z_{3}\right) \in
U\times U\times U$ to be a three-periodic orbit of $B$. The analogue of
Tabachnikov's method using the Morse and Lusternik-Schnirelman theory, would
be that the triplet $\left( p_{1},p_{2},p_{3}\right) \in N\times N\times N$
consisting of the middle points of the geodesic triangle determined by $%
\left( z_{1},z_{2},z_{3}\right) $ is a critical point of the function $A=_{%
	\text{def}}\left. \mathcal{A}\right\vert _{N\times N\times N}$. We found a
geometric condition for $\left( p_{1},p_{2},p_{3}\right) $ to be so,
involving $J\left( \nu \left( p_{i}\right) \right) $, which unfortunately we
were able to prove true only if the triplets lie in a totally geodesic
surface. So, we were not able to close the
argument. We mention that in between we needed and found a concise
expression for the differential of $A$ (in terms of Jacobi fields).

\section{Preliminaries}

First we introduce the complex hyperbolic plane. The general references 
for its definition and geometric properties are \cite{refCH2} and \cite{BridsonHaefliger}.

Let $\mathbb{C}H^{2}$ be the two-dimensional complex hyperbolic space
equipped with the Fubini--Study metric $\langle \cdot ,\cdot \rangle $ of
constant holomorphic sectional curvature $-4$. It is the unique
four-dimensional simply connected complex complete Riemannian manifold whose complex structure $J$ and curvature tensor $R$ are parallel (in particular, it is a K\"{a}hler symmetric space), such that
\begin{equation}
	\langle R(Ju,u)u,Ju\rangle =-4\,\,\,\,\,\,\text{and}\,\,\,\,\,\,\langle
	R(v,u)u,v\rangle =-1  \label{Curvature}
\end{equation}%
for any unit vectors $u,v\in T_{p}\mathbb{C}H^{2}$ with $v$ perpendicular to both $u$ and $J(u)$, 
for all $p\in \mathbb{C}H^{2}$.

It is well known that Jacobi fields along a unit speed geodesic $\gamma $,
which appear as geodesic variations of $\gamma $, are exactly those vector
fields $K$ along $\gamma $ satisfying the equation 
\begin{equation}
	\dfrac{D^{2}K}{dt^{2}}+R(K,\gamma ^{\prime })\gamma ^{\prime }=0\text{.}
	\label{Jacobi}
\end{equation}

\begin{lemma}
	\label{lemmaJF} Let $K$ be a Jacobi field along a unit speed geodesic $%
	\gamma $ in $\mathbb{C}H^{2}$ such that 
	\begin{equation*}
		K(0)=a\gamma ^{\prime }(0)+v \text{\ \ \ \ \ \ and \ \ \ \ \ \ }
		\tfrac{DK}{dt}(0)=w+bJ\gamma ^{\prime }(0), 
	\end{equation*}
	with $a,b\in \mathbb{R}$ and $%
	v,w $ perpendicular to both $\gamma ^{\prime }(0)$ and $J\gamma ^{\prime }(0)$. Then 
	\begin{equation}
		K(r)=a\gamma ^{\prime }(r)+\cosh (r)V(r)+\tfrac{b}{2}\sinh (2r)J\gamma
		^{\prime }(r)+\sinh (r)W(r),  \label{CampoJacobi}
	\end{equation}%
	where $V,W$ are the parallel vector fields along $\gamma $ with initial
	conditions $V(0)=v$ and $W(0)=w$.
\end{lemma}

\begin{proof}
	Since $J\circ \gamma ^{\prime }$, $V$ and $W$ are parallel vector fields
	along $\gamma $ and $V(r)$ and $W(r)$ are perpendicular 
	to both $\gamma ^{\prime }(r)$ and $J\gamma ^{\prime }(r)$ for all $r$, we have by (\ref{Curvature}) that
	\begin{equation*}
		R(J\gamma ^{\prime }(r),\gamma ^{\prime }(r))\gamma ^{\prime }(r)=-4J\gamma
		^{\prime }(r),
	\end{equation*}%
	\begin{equation*}
		R(V(r),\gamma ^{\prime }(r))\gamma ^{\prime }(r)=-V(r)\hspace{0.3cm}\text{and%
		}\hspace{0.3cm}R(W(r),\gamma ^{\prime }(r))\gamma ^{\prime }(r)=-W(r).
	\end{equation*}
	Now, a straightforward computation yields that the vector field in (\ref%
	{CampoJacobi}) satisfies the Jacobi equation (\ref{Jacobi}), as desired.
\end{proof}

\bigskip

Let $N$ be a quadratically convex compact connected oriented hypersurface of $\mathbb{C}H^{2}$
and let $\nu $ be the unit normal vector field pointing inwards, that is, in
the direction of the mean curvature vector field of $N$.

For each $p\in N$, we denote by $S_{p}:T_{p}N\rightarrow T_{p}N$ the shape
operator of $N$, that is, 
\begin{equation}
	S_{p}(x)=-\nabla _{x}\,\nu ,  \label{shapeO}
\end{equation}%
for $x\in T_{p}N$, where $\nabla $ denotes the Levi-Civita connection of $%
\mathbb{C}H^{2}$.

We recall the definition of the 
smooth map
\begin{equation*}
	F:N\times \mathbb{R}\rightarrow \mathbb{C}H^{2}\text{, \ \ \ \ \ \ \ }F(q,r)=\gamma _{J\nu (q)}(r)\text{,}
\end{equation*}%
for $q\in N$ and $r\in \mathbb{R}$, where $\gamma _{v}$ denotes as before the unique
geodesic on $\mathbb{C}H^{2}$ with initial velocity $v$. 

\begin{lemma}
	\label{K_x} Let $x\in T_{p}N$. Then, 
	\begin{equation*}
		dF_{(p,t)}(x,0)=K_{x}(t)
	\end{equation*}
	for all $t\in \mathbb{R}$, where $K_{x}$ is the Jacobi field along $\gamma
	_{J\nu (p)}$ with initial conditions 
	\begin{equation}
		K_{x}(0)=x\hspace{1cm}\text{and}\hspace{1cm}\tfrac{DK_{x}}{dr}(0)
		=-(J\circ S_{p})(x)\text{.}  \label{condK_x}
	\end{equation}
\end{lemma}

\begin{proof}
	Let $\alpha $ be a smooth curve in $N$ whose initial velocity is $x$. We
	consider the geodesic variation of $\gamma _{J\nu (p)}$ given by $%
	(s,r)\rightarrow F(\alpha (s),r)$ for $(s,r)\in (-\varepsilon ,\varepsilon
	)\times \mathbb{R}$. Let us see that the corresponding Jacobi vector field $%
	K $ along $\gamma _{J\nu (p)}$ is $K_{x}$. We compute
	\begin{equation*}
		K\left( r\right) =\left. \tfrac{d}{ds}\right\vert _{0}F(\alpha (s),r)=\left. 
		\tfrac{d}{ds}\right\vert _{0}\gamma _{J\nu (\alpha \left( s\right) )}(r).
	\end{equation*}%
	Hence $K\left( 0\right) =\alpha ^{\prime }\left( 0\right) =x$ and%
	\begin{eqnarray*}
		\tfrac{DK}{dr}(0) &=&\left. \tfrac{D}{dr}\right\vert _{0}\left. \tfrac{d}{ds}%
		\right\vert _{0}\gamma _{J\nu (\alpha \left( s\right) )}(r)=\left. \tfrac{D}{%
			ds}\right\vert _{0}\left. \tfrac{d}{dr}\right\vert _{0}\gamma _{J\nu (\alpha
			\left( s\right) )}(r) \\
		&=&\left. \tfrac{D}{ds}\right\vert _{0}J\nu (\alpha \left( s\right) )=\nabla
		_{x}J\nu \\
		&=& J\left( \nabla _{x}\,\nu \right) =-(J\circ S_{p})(x),
	\end{eqnarray*}%
	since $J$ is parallel (see (\ref{shapeO})).
\end{proof}

\section{Proofs of the results}

Before proving Proposition \ref{Examples} we give details of the definitions
and facts that precede its statement. This is the only part of the article
where we use a model for $\mathbb{C}H^{2}$. We consider the canonical
Hermitian inner product on $\mathbb{C}^{2}$, that is, $\left\langle \left(
z_{1},z_{2}\right) ,\left( w_{1},w_{2}\right) \right\rangle =\bar{z}%
_{1}w_{1}+\bar{z}_{2}w_{2}$, and denote as usual 
$\left\vert z\right\vert ^{2}=\left\langle z,z\right\rangle $. 
Let $\mathbb{B}=\left\{p\in \mathbb{C}^{2}\mid \left\vert p\right\vert ^{2}<1\right\} $ 
be the ball model of $\mathbb{%
	C}H^{2}$, where $\mathbb{B}$ is endowed with the Riemannian metric with distance $d$
given by%
\begin{equation}
	\cosh ^{2}d\left( p,q\right) =\frac{\left( 1-\left\langle p,q\right\rangle
		\right) \left( 1-\left\langle q,p\right\rangle \right) }{\left( 1-\left\vert
		p\right\vert ^{2}\right) \left( 1-\left\vert q\right\vert ^{2}\right) }
	\label{distance}
\end{equation}%
(see 10.24 in \cite{BridsonHaefliger}). The restriction $I$ of the canonical
complex structure of $\mathbb{C}^{2}$ (that is, multiplication by $i$)
induces the complex structure on $\mathbb{B}$.

By (\ref{distance}), the map $I$ and the reflection $\rho :\mathbb{%
	B\rightarrow B}$, $\rho \left( z,w\right) =\left( z,-w\right) $, are
isometries. Now, $\left( \mathbb{C\times }\left\{ 0\right\} \right) \cap 
\mathbb{B}$ is totally geodesic, since it is a connected set of fixed points
of an isometry (see for instance Proposition 10.3.6 in \cite{Carlos}). By (%
\ref{Curvature}), it is isometric to the hyperbolic plane with (constant)
curvature $-4$.

\bigskip

\begin{proof}[Proof of Proposition \protect\ref{Examples}]
	a) Since the group of orientation preserving isometries of $\mathbb{C}H^{2}$
	acts transitively on the set of complex lines, we may assume without loss of
	generality that $H=\left( \mathbb{C\times }\left\{ 0\right\} \right) \cap 
	\mathbb{B}$.
	
	Let $\nu $ be the unit normal vector field on $N$ pointing inwards
	and let $p\in N\cap H$.
	Since by hypothesis $N$ is invariant by the reflection $\rho $, we have that 
	$d\rho _{p}\left( T_{p}N\right) =T_{p}N$. In particular, $d\rho _{p}\left(
	\nu \left( p\right) \right) =\varepsilon \nu \left( p\right) $ with $%
	\varepsilon =\pm 1$, since $\rho $ is an isometry. But $\rho $ preserves the
	mean curvature vector at $p$, hence $\varepsilon =1$. Thus, $\nu \left(
	p\right) \in T_{p}H$ since $T_{p}H=\left\{ w\in T_{p}\mathbb{C}H^{2}\mid
	d\rho _{p}\left( w\right) =w\right\} $. Now, the fact that $H$ is a complex
	line implies that $J\left( \nu \left( p\right) \right) $ is also in $T_{p}H$%
	. Finally, as $H$ is totally geodesic, by definition of the billiard map $B$
	on $U$, $B$ preserves $H\cap U$ and coincides with the usual outer billiard
	map in $H$ associated with the curve of positive geodesic curvature obtained by
	intersecting $N$ with $H$.
	
	\smallskip
	
	b) We may suppose that the geodesic sphere is centered at the origin of $%
	\mathbb{B}$. By (\ref{distance}), it is a Euclidean sphere, and so invariant
	by the reflection with respect to any complex line through the origin.
\end{proof}

\bigskip

\begin{proof}[Proof of Theorem \protect\ref{F+F-}]
	We prove that $F_{+}$ is a diffeomorphism (the case $F_{-}$ is similar). 
	After showing that the image is contained in the exterior of $N$, we prove
	that it is a local diffeomorphism using the Inverse Function Theorem and obtain
	the global condition with a topological argument.
	
	First, we observe that the image of $F_{+}$ is contained on $U$. Indeed,
	suppose that there exists $(p,t)\in N\times (0,\infty )$ such that $%
	F_{+}(p,t)=q\notin U$. Since $N$ is embedded in $\mathbb{C}H^{2}$ as the
	boundary of a convex body (see Theorem 1 in \cite{Alex}), the geodesic
	segment joining $p$ with $q$, that is, the restriction of $\gamma _{J\nu
		_{p}}$ to $\left[ 0,t\right] $, lies in this body. On the other hand, since $%
	N$ is quadratically convex, we have that $\gamma _{J\nu _{p}}(s)$ is in $U$ for
	sufficiently small $s>0$, a contradiction.
	
	Now, we fix $p\in N$ and $t>0$. We will see that $(dF_{+})_{(p,t)}$ is an
	isomorphism. Let $u$ be a unit tangent vector of $N$ at $p$ orthogonal to $%
	J\nu (p)$. We consider the basis of $T_{(p,t)}(N\times (0,\infty ))$ given
	by 
	\begin{equation*}
		\mathcal{C}_{t}=\left\{ \left( 0,\left. \tfrac{\partial }{\partial s}%
		\right\vert _{t}\right) ,(J\nu (p),0),(u,0),(Ju,0)\right\} .
	\end{equation*}%
	Besides, we call $\mathcal{B}_{t}$ the basis of $T_{\gamma _{J\nu (p)}(t)}%
	\mathbb{C}H^{2}$ obtained by the parallel transport of the basis $\{\nu
	(p),J\nu (p),u,Ju\}$ along $\gamma _{J\nu (p)}$, between $0$ and $t$, that
	is, 
	\begin{equation}
		\mathcal{B}_{t}=\{\tau _{0}^{t}(\nu (p)),\tau _{0}^{t}(J\nu (p)),\tau
		_{0}^{t}(u),\tau _{0}^{t}(Ju)\}\text{.}  \label{baseBt}
	\end{equation}%
	We compute the matrix of $(dF_{+})_{(p,t)}$ with respect to the bases $%
	\mathcal{C}_{t}$ and $\mathcal{B}_{t}$. By Lemma \ref{K_x} we have that $%
	dF_{(p,t)}(x,0)=K_{x}(t)$ for all $x\in T_{p}N$, hence, we need to compute
	the Jacobi fields $K_{J\nu \left( p\right) }$, $K_{u}$ and $K_{Ju}$.
	
	We denote by $A$ the (symmetric) $(3\times 3)$-matrix of the shape operator $%
	S_{p}$ at $p$ with respect to the orthonormal basis $\{J\nu (p),u,Ju\}$ of $%
	T_{p}N$. By the usual abuse of notation, we put $K^{\prime }=\frac{DK}{dr}$.
	By (\ref{condK_x}), we have that
	\begin{equation*}
		\begin{array}{ccc}
			K_{J\nu _{p}}^{\prime }(0) & = & A_{11}\nu _{p}-A_{12}Ju+A_{13}u, \\ 
			K_{u}^{\prime }(0) & = & A_{12}\nu _{p}-A_{22}Ju+A_{23}u, \\ 
			K_{Ju}^{\prime }(0) & = & A_{13}\nu _{p}-A_{23}Ju+A_{33}u.
		\end{array}%
	\end{equation*}
	
	Using (\ref{CampoJacobi}) we can compute the Jacobi fields $K_{J\nu _{p}}$, $%
	K_{u}$ and $K_{Ju}$. Since $K_{x}\left( 0\right) =x$ and $J\circ \,\tau
	_{0}^{t}=\tau _{0}^{t}\,\circ J$, in terms of the basis $\mathcal{B}_{t}$ we
	obtain%
	\begin{equation*}
		\begin{array}{c}
			dF_{(p,t)}(J\nu _{p},0)=A_{11}\cosh t\sinh t\,\tau _{0}^{t}(\nu _{p})+\tau
			_{0}^{t}(J\nu _{p})+A_{13}\sinh t\,\tau _{0}^{t}(u)-A_{12}\sinh t\,\tau
			_{0}^{t}(Ju), \\ 
			\\ 
			dF_{(p,t)}(u,0)=A_{12}\cosh t\sinh t\,\tau _{0}^{t}(\nu _{p})+(\cosh
			t+A_{23}\sinh t)\,\tau _{0}^{t}(u)-A_{22}\sinh t\,\tau _{0}^{t}(Ju), \\ 
			\\ 
			dF_{(p,t)}(Ju,0)=A_{13}\cosh t\sinh t\,\tau _{0}^{t}(\nu _{p})+A_{33}\sinh
			t\,\tau _{0}^{t}(u)+(\cosh t-A_{23}\sinh t)\,\tau _{0}^{t}(Ju).%
		\end{array}%
	\end{equation*}%
	On the other hand, it is easy to see that $dF_{(p,t)}\left( 0,\left. \frac{%
		\partial }{\partial s}\right\vert _{t}\right) =\tau _{0}^{t}(J\nu _{p})$.
	Now, putting $D_{t}=[(dF_{+})_{(p,t)}]_{\mathcal{C}_{t},\mathcal{B}_{t}}$ we
	have 
	\begin{equation}
		D_{t}=%
		\begin{pmatrix}
			0 & A_{11}\cosh t\sinh t & A_{12}\cosh t\sinh t & A_{13}\cosh t\sinh t \\ 
			1 & 1 & 0 & 0 \\ 
			0 & A_{13}\sinh t & \cosh t+A_{23}\sinh t & A_{33}\sinh t \\ 
			0 & -A_{12}\sinh t & -A_{22}\sinh t & \cosh t-A_{23}\sinh t%
		\end{pmatrix}%
		\text{.}  \label{D_t}
	\end{equation}%
	A straightforward computation yields 
	\begin{equation*}
		\det D_{t}=-\cosh t\sinh t(\cosh
		^{2}t~A_{11}+\sinh ^{2}t\det A), 
	\end{equation*}
	which is different from zero since $t>0$ 
	and $A_{11}$ and $\det A$ are positive, as $N$ is quadratically convex
	and $\nu $ points inwards. Therefore, $dF_{(p,t)}$ is an isomorphism for any 
	$(p,t)$ and thus
	$F_{+}$ is a local diffeomorphism by the Inverse Function Theorem.
	
	Next we show that $F_{+}:N\times (0,\infty )\rightarrow U$ is proper. Let $K$
	be a compact set in $U$. Let $r_{0}=d\left( N,K\right) $ be the distance between
	$N$ and $K$ 
	and $r_{1}=\max
	\left\{ d\left( p,q\right) \mid p\in N,q\in K\right\} $. Let us see that $%
	\left( F_{+}\right) ^{-1}\left( K\right) $ is a closed set contained in $%
	N\times \left[ r_{0},r_{1}\right] $ and so it is compact. Indeed, suppose
	that $F_{+}\left( q,t\right) =\gamma _{J\nu \left( q\right) }\left( t\right)
	\in K$. Since $\mathbb{C}H^{2}$ is a Hadamard manifold, the distance between 
	$q=\gamma _{J\nu \left( q\right) }\left( 0\right) $ and $F_{+}\left(
	q,t\right) $ is equal to $t$. Hence, $r_{0}\leq t\leq r_{1}$.
	
	Since $F_{+}$ is a proper local homeomorphism, it is a covering map 
	(see for instance Problem 11-9 in \cite{Lee}). Now, $N$
	is homeomorphic to $S^{3}$ by \cite{Alex}, and so $U$ (the exterior) is
	simply connected. Therefore $F_{+}$ is a diffeomorphism, as desired.
\end{proof}

\bigskip

\begin{proof}[Proof of Theorem \protect\ref{Bdifeomorph}]
	We observe that $B=F_{+}\circ g\circ (F_{-})^{-1}$, where 
	\begin{equation*}
		g:N\times (-\infty ,0)\rightarrow N\times (0,\infty ),\hspace{1cm}%
		g(p,t)=(p,-t)
	\end{equation*}%
	is a diffeomorphism. Hence, $B$ is a diffeomorphism by Theorem \ref{F+F-}.
\end{proof}

\bigskip

We consider the associated fundamental symplectic form $\omega $ defined by $%
\omega (\cdot ,\cdot )=\langle J\cdot ,\cdot \rangle $, where $J$ is the
complex structure on $\mathbb{C}H^{2}$. We fix $t\neq 0$ and compute the
matrix of $\omega $ with respect to $\mathcal{B}_{t}$ in (\ref{baseBt}).
Using, as before, that $J$ is parallel, it is not difficult to see
that $[\omega ]_{\mathcal{B%
	}_{t}}=$ diag$\left( j,j\right) $, where $j=%
\begin{pmatrix}
	0 & -1 \\ 
	1 & 0%
\end{pmatrix}%
$.

\bigskip

\begin{proof}[Proof of Theorem \protect\ref{TeoSymplectic}]
	Given $q\in U$, by Theorem \ref{F+F-} 
	there exists a unique $(p,t)\in N\times (0,\infty )$ such
	that $q=F_{-}(p,-t)$. Calling $H_{t}=[(dB)_{q}]_{\mathcal{B}_{-t},\,\mathcal{%
			B}_{t}}$, we want to show that 
	\begin{equation}
		H_{t}^{T}\,[\omega ]_{\mathcal{B}_{t}}\,H_{t}=[\omega ]_{\mathcal{B}_{-t}}
		\label{Bsimplect}
	\end{equation}%
	for all $p\in N$ and $t>0$. By the proof of Theorem \ref{Bdifeomorph} we
	have that%
	\begin{equation*}
		H_{t}=D_{t}\,G\,(D_{-t})^{-1}\text{,}
	\end{equation*}
	where $G=[(dg)_{(-t,p)}]_{\mathcal{C}_{-t},\,\mathcal{C}_{t}}=\text{diag}%
	\left( -1,1,1,1\right) $ and $D_{t}$ is as in (\ref{D_t}). Hence, (\ref%
	{Bsimplect}) holds if and only if 
	\begin{equation}
		G\,D_{t}^{T}[\omega ]_{\mathcal{B}_{t}}\,D_{t}G=D_{-t}^{T}\,[\omega ]_{%
			\mathcal{B}_{-t}}\,D_{-t}.  \label{Bsimplect2}
	\end{equation}
	
	For $\delta =\pm 1$, we call $E(\delta t)=D_{\delta t}^{T}\,[\omega ]_{%
		\mathcal{B}_{\delta t}}\,D_{\delta t}$. So, (\ref{Bsimplect2}) translates
	into $GE(t)G=E(-t)$. Since $[\omega ]_{\mathcal{B}_{\delta t}}=$ diag$\left(
	j,j\right) $ and $E\left( \delta t\right) $ is skew-symmetric, a
	straightforward computation yields%
	\begin{equation*}
		\begin{array}{lll}
			E_{21}(\delta t) & = & -\delta A_{11}\cosh t\sinh t, \\ 
			E_{31}(\delta t) & = & -\delta A_{12}\cosh t\sinh t, \\ 
			E_{41}(\delta t) & = & -\delta A_{13}\cosh t\sinh t, \\ 
			E_{32}(\delta t) & = & (A_{12}A_{23}-A_{13}A_{22})\sinh ^{2}t, \\ 
			E_{42}(\delta t) & = & (A_{12}A_{33}-A_{13}A_{23})\sinh ^{2}t, \\ 
			E_{43}(\delta t) & = & (A_{22}A_{33}-A_{23}^{2})\sinh ^{2}t+\cosh ^{2}t\text{%
				,}%
		\end{array}%
	\end{equation*}%
	from which $GE(t)G=E(-t)$ follows for all $t>0$, as desired.
\end{proof}

\vspace{0.5cm}

\noindent Yamile Godoy. Email: yamile.godoy@unc.edu.ar.

\noindent Marcos Salvai. Email: marcos.salvai@unc.edu.ar

\medskip

\noindent FAMAF (Universidad Nacional de C\'ordoba) and CIEM (Conicet),

\noindent Av.\ Medina Allende s/n, Ciudad Universitaria, 

\noindent X5000HUA C\'ordoba,
Argentina.


\begin{thebibliography}{99}
\bibitem{TabachNuevo} P. Albers, A. Ch\'avez-C\'aliz, S. Tabachnikov,
\textsl{Outer symplectic billiards}, arXiv:2409.07990 math.SG

\bibitem{Alex} S. Alexander, \textsl{Locally convex hypersurfaces of negatively
	curved spaces}, Proc. Am. Math. Soc. 64 (1977) 321--325.

%\bibitem{Brehm} Geom. Dedicata 33, No. 1, 59-76 (1990).

\bibitem{refCH2} J. Berndt, F. Tricerri, L. Vanhecke, \textsl{Generalized Heisenberg groups and Damek–Ricci harmonic
	spaces}, Lecture Notes in Mathematics 1598, Springer, Berlin, 1995.

\bibitem{Carlos} J. Berndt, S. Console, C. Olmos, \textsl{Submanifolds and
	Holonomy}, Monographs and Research Notes in Mathematics, CRC
Press, Boca Raton FL, 2016.

\bibitem{BridsonHaefliger} M.R. Bridson, A. Haefliger, \textsl{Metric
	spaces of non-positive curvature}, Grundlehren der mathematischen
Wissenschaften 319, Springer, Berlin\thinspace -\thinspace New York, 1999.

\bibitem{TabachMathIntell} F. Dogru, S. Tabachnikov, 
\textsl{Dual billiards}, Math. Intelligencer 27 (4) (2005) 18--25.

%\bibitem{Triangles} Boumediene Et-Taoui, On the medians of a triangle in
%two-point homogeneous spaces.

\bibitem{Planar} R. Karasev, A. Sharipova, \textsl{Convex bodies with all characteristics planar}, 
Int. Math. Res. Not. 2024 (10) (2024) 8104--8121.

%\bibitem{Lima} Ronaldo Freire de Lima, A Survey on Convex Hypersurfaces of
%Riemannian Manifolds, Matem\'{a}tica Contempor\^{a}nea, 2022, Special Issue
%in honor of Professor Renato de Azevedo Tribuzy on the occasion of his 75th
%birthday.

\bibitem{Ibero} Y. Godoy, M. Harrison, M. Salvai, \textsl{Tangent ray
	foliations and their associated outer billiards}, Rev. Mat. Iberoam. 39 (6)  (2023) 2349--2369.

\bibitem{London} Y. Godoy, M. Harrison, M. Salvai, \textsl{Outer
	billiards in the spaces of oriented geodesics of the three-dimensional space
	forms}, J. Lond. Math. Soc., II. Ser. 109 (6), Article ID e12922, 25 p.
(2024).

\bibitem{Lee} J.M. Lee, \textsl{Introduction to topological manifolds}. Second edition. 
Graduate Texts in Mathematics 202, Springer, New York, 2011.

\bibitem{moser1} J. Moser, \textsl{Stable and random motions in dynamical
	systems}. Ann. of Math. Stud. 77, Princeton, 1973.

\bibitem{moser2} J. Moser, \textsl{Is the solar system stable?}, Math.
Intelligencer 1 (1978) 65--71.

\bibitem{Neumann} B.H. Neumann, \textsl{Sharing Ham and Eggs}, Iota:
The Manchester University Mathematics Students' Journal, January 1959.

\bibitem{TabachnikovOuter} S. Tabachnikov, \textsl{Outer billiards}, Russian
Math. Surveys 48 (6) (1993) 81--109.

\bibitem{TabachnikovBilliards} S. Tabachnikov, \textsl{Billiards}. Panoramas
et Synth\`{e}ses 1, Socit\'e Math\'ematique de France, 1995.

\bibitem{Tabachnikov} S. Tabachnikov, \textsl{On the dual billiard problem},
Adv. Math. 115 (2) (1995) 221--249.

\bibitem{TabachHypPlane} S. Tabachnikov, \textsl{Dual billiards in the hyperbolic
	plane}, Nonlinearity 15 (4) (2002) 1051--1072.

\bibitem{TabachPeriodic} S. Tabachnikov, \textsl{On three-periodic trajectories
	of multi-dimensional dual billiards}, Algebr. Geom. Topol. 3 (2003) 993--1004.

\bibitem{bookT} S. Tabachnikov, \textsl{Geometry and billiards}. Student
Mathematical Library 30, American Mathematical Society, Providence RI, 2005.\end{thebibliography}
\end{document}